\documentclass[]{article}
\usepackage{mathtools,wasysym}

\usepackage{amscd,amssymb,amsthm,graphicx,multicol,eucal,amsmath,enumerate,url}
\usepackage{xcolor}

\def \2{\^{a}}
\def \1{\^{i}}

\newcommand{\C}{{\mathbb  C}}
\numberwithin{equation}{section}
\newtheorem{thm}{\bf Theorem}[section]
\newtheorem{lem}[thm]{\bf Lemma}

\newtheorem{cor}[thm]{\bf Corollary}
\newtheorem{defn}{\bf Definition}[section]
\theoremstyle{remark}
\newtheorem{rem}{\bf Remark}[section]

\newtheorem{exmp}{\bf Example}[section]

\begin{document}

\pagenumbering{arabic}

\author{Dan Com\u anescu\\
{\small Department of Mathematics, West University of Timi\c soara}\\
{\small Bd. V. P\^ arvan, No 4, 300223 Timi\c soara, Rom\^ ania}\\
{\small E-mail address: dan.comanescu@e-uvt.ro}}
\title{\textbf {Linear matrix equations with parameters forming a commuting set of diagonalizable matrices}}
\date{}
\maketitle

\begin{abstract}
 We prove that the following statements are equivalent: a linear matrix equation with parameters forming a commuting set of diagonalizable matrices is consistent, a certain matrix constructed with the Drazin inverse is a solution of this matrix equation, the attached standard linear matrix equation is consistent. The number of zero components of a given matrix (the relevant matrix) gives us the dimension of the affine space of solutions. We apply this theoretical results to the Sylvester equation, the Stein equation, and the Lyapunov equation. We present a description of the set of the diagonalizing matrices for a commuting sequence of diagonalizable matrices.
 \end{abstract}

\noindent {\bf Keywords:} linear matrix equation, Sylvester equation, Stein equation, Lyapunov equation, diagonalizable matrix, normal matrix, diagonalizing matrix, Drazin inverse, Moore-Penrose inverse \\
{\bf MSC Subject Classification 2020:} 15A06, 15A09, 15A10, 15A20, 15A24, 15A27.

\section{Introduction}

A significant number of real-world phenomena can be represented as mathematical models whose study includes the determination of solutions of a linear matrix equation. In the second half of the 20th century, significant efforts were made to study the continuous Lyapunov equation and the discrete Lyapunov equation. These appear with the use of Lyapunov's second method in control theory (see \cite{parks}, \cite{gajic}). Sylvester's equation generalizes the continuous Lyapunov's equation. This is named after the mathematician J. J. Sylvester who at the end of the 19th century, in the paper \cite{sylvester}, studied the homogeneous case. The discrete Lyapunov equation is a special case of the Stein equation. Due to the ever-increasing dimensions of the matrices involved in the matrix equations, it was necessary to find some numerical algorithms that would allow the determination or approximation of the solutions of these equations. The presentation of such numerical methods can be found in the papers \cite{golub}, \cite{simoncini}, \cite{jbilou}, \cite{zhang}. 

All matrix equations shown above (with square matrices) are particular cases of a linear matrix equation of the following form
\begin{equation}\label{linear-matrix-equation}
\sum_{j=1}^kA_jXB_j=C,
\end{equation}
with the unknown matrix $X\in \mathcal{M}_n(\C)$ and the parameters $k\in \mathbb{N}^*$ and $A_1, \dots, A_k$, $B_1, \dots, B_k$, $C\in \mathcal{M}_n(\C)$. 

Although in particular situations solutions have been presented, see \cite{lancaster}, in the general case solving this equation is extremely difficult.

In this paper we study the matrix equation \eqref{linear-matrix-equation} in the situation where $\mathcal{N}=\{A_1, \dots, A_k, B_1, \dots, B_k, C\}$ is a commuting set of diagonalizable matrices. Our study starts from the observation that $\mathcal{N}$ is a simultaneously diagonalizable set. It is obvious that, in this case, the set $\mathfrak{S}(\mathcal{N})$ of all diagonalizing matrices for $\mathcal{N}$ is important.  

Section \ref{main-section} presents the main theoretical results. Using the eigenvalues of the matrices $A_1$, ..., $A_k$, $B_1$, ..., $B_k$ we build $\Gamma^S$, the so-called relevant matrix with respect to $S\in\mathfrak{S}(\mathcal{N})$, and we prove that any other relevant matrix permutes the rows and the columns of $\Gamma^S$ in the same way. The number of components equal to zero of the relevant matrix gives us the dimension of the affine space of the solutions of \eqref{linear-matrix-equation}. When we know the matrix $S\in \mathfrak{S}(\mathcal{N})$ we can easily calculate the relevant matrix. Since specifying a matrix $S\in \mathfrak{S}(N)$ can be difficult, it is important to study the situations where we can compute a relevant matrix without a prior computation of the matrix $S$. 
The consistency of the linear matrix equation \eqref{linear-matrix-equation} is equivalent to the consistency of the attached standard linear matrix equation \eqref{standard-equation}. 
It is also equivalent with the fact that the matrix constructed in \eqref{X-hat-99}, using the Drazin inverse of a matrix, is a solution of the equation  \eqref{linear-matrix-equation} or of the equation \eqref{standard-equation}. At the end of this section we consider the special case where all matrices in $\mathcal{N}$ are normal matrices.  

In Section \ref{section-Sylvester}  are adapted the results from Section \ref{main-section} to the particular cases of Sylvester equation respectively continuous Lyapunov equation. Section \ref{section-Stein} applies the results of  Section \ref{main-section} to the Stein equation, respectively the discrete Lyapunov equation. 

In Section \ref{diagonalizing-matrices-654} is made a study of the diagonalizing matrices for a commuting sequence of diagonalizable matrices $\mathcal{M}$. 
At the end of this paper we present some examples and we mention some notions and results about matrix theory. 

The importance of linear matrix equations led to the study of some of their generalizations. The paper \cite{sari} study the discrete Lyapunov equation in a semiring. A generalization of the above linear matrix equation to a tensor equation is shown in \cite{xu}. The paper \cite{gil} is devoted to the discrete Lyapunov equation in a Hilbert space.

\section{Main results}\label{main-section}

In this section we study the solutions of the linear matrix equation \eqref{linear-matrix-equation} when $\mathcal{N}=\{A_1, \dots, A_k, B_1, \dots, B_k, C\}$ is a commuting set of diagonalizable matrices; for all $j,l\in \{1,\dots, k\}$ we have $A_jB_l=B_lA_j$, $A_jC=CA_j$, $B_lC=CB_l$.

First, we present some notions and results about the diagonalizable matrices.  $M\in \mathcal{M}_n(\C)$ is a diagonalizable matrix if there exists\footnote{$\mathcal{M}^{\text{inv}}_n(\C)$ is the subset of $\mathcal{M}_n(\C)$ formed with the invertible matrices.} $S\in \mathcal{M}^{\text{inv}}_n(\C)$  such that $S^{-1}MS$ is a diagonal matrix. The matrix $S$ is called a diagonalizing matrix for $M$. The set of all diagonalizing matrices for $M$ is noted by $\mathfrak{S}(M)$. For the diagonalizable matrix $M\in \mathcal{M}_n(\C)$ we consider the function $\Delta_M:\mathfrak{S}(M)\to \C^n$ given by\footnote{If ${\bf m}=\text{vec}(m_1, \dots, m_n)\in \C^n$, then $\text{diag}({\bf m})\in \mathcal{M}_n(\C)$ has $m_1,\dots, m_n$ on the main diagonal and the entries outside the main diagonal are all zero.} 
$\Delta_M(S)={\bf m}$, where $\text{diag}({\bf m})=S^{-1}MS$.
We say that the vector ${\bf m}$ is induced by the diagonalizable matrix $M$.  
In what follows we will use the following notations. 
\begin{defn} Let be ${\bf m}\in \C^n$ and the permutation $\sigma\in \mathcal{S}_n$.
\begin{enumerate}[(i)]
\item If ${\bf m}=\emph{vec}(m_1,\dots,m_n)$, then  ${\bf m}_{\sigma}=\emph{vec}(m_{\sigma(1)},\dots, m_{\sigma(n)})\in \C^n$.

\item ${\bf m}$ is a star vector if it has the form\footnote{${\bf 1}_m=(\stackrel{m\,-\text{times}}{\overbrace{1,1,\dots,1}})$.} ${\bf m}=\emph{vec}(\lambda_1 {\bf 1}_{k_1}, \dots, \lambda_d {\bf 1}_{k_d})$, where $\lambda_1,\dots,\lambda_d$ are different complex numbers and $k_1+ \dots +k_d=n$. 
\end{enumerate}
\end{defn}
\noindent Using these notations we have $\text{diag}({\bf m}_{\sigma})=P_{\sigma}^{-1}\text{diag}({\bf m})P_{\sigma}$, where 
$P_{\sigma}$ is the permutation matrix generated by the permutation $\sigma\in \mathcal{S}_n$ (see Appendix \ref{permutation-matrix-22}). We present some useful results related to the above notions.

\begin{lem}\label{U-P-234}
Let $M\in \mathcal{M}_n(\C)$ be a diagonalizable matrix and ${\bf m}=\Delta_M(S)$ with $S\in \mathfrak{S}(M)$. The following statements hold.
\begin{enumerate}[(i)] 
\item $SP_{\sigma}\in \mathfrak{S}(M)$ and $\Delta_M(SP_{\sigma})={\bf m}_{\sigma}$, where $\sigma\in \mathcal{S}_n$.
\item $\Delta_M(\mathfrak{S}(M))=\{{\bf m}_{\sigma}\,|\,\sigma\in \mathcal{S}_n\}.$
\item If $\lambda_1, \dots, \lambda_d$ are the distinct eigenvalues of $M$ and $k_1, \dots, k_d$ are their algebraic multiplicity, then ${\bf m}=\emph{vec}(\lambda_1 {\bf 1}_{k_1}, \dots, \lambda_d {\bf 1}_{k_d})\in \C^n$ is a star vector induced by $M$.
\end{enumerate}
\end{lem}

\begin{proof} $(i)$ We have 
$P_{\sigma}^{-1}S^{-1}MSP_{\sigma}=P_{\sigma}^{-1}\text{diag}({\bf m})P_{\sigma}=\text{diag}({\bf m}_{\sigma})$.
For $(ii)$ and $(iii)$ we observe that all vectors that are induced by the diagonalizable matrix $M$ have the eigenvalues of $M$ as components. 
\end{proof}
The structure of $\mathfrak{S}(M)$ can be specified starting with the choice of an induced star vector, see Lemma \ref{horn-1-3-27}.

If $\mathcal{M}$ is a set of diagonalizable matrices, then the set of all diagonalizing matrices for $\mathcal{M}$ is 
$\mathfrak{S}(\mathcal{M})=\bigcap_{M\in \mathcal{M}}\mathfrak{S}(M)$.
$\mathfrak{S}(\mathcal{M})$ is a non-empty set if and only if $\mathcal{M}$ is a commuting set of diagonalizable matrices (\cite{horn}, Theorem 1.3.21).
If $\mathcal{M}:=(M_1, M_2,\dots, M_q)$ is a commuting sequence of diagonalizable matrices, then a sequence of induced vectors by $\mathcal{M}$ has the form $(\Delta_{M_1}(S), \Delta_{M_2}(S), \dots, \Delta_{M_q}(S))\in (\C^n)^q,$ where $S\in \mathfrak{S}(\mathcal{M})$. The structure of $\mathfrak{S}(\mathcal{M})$ is presented in Theorem \ref{description-s(M)-321}.

The standard linear matrix equation attached to \eqref{linear-matrix-equation} is
\begin{equation}\label{standard-equation}
\Big(\sum_{j=1}^kA_jB_j\Big)X=C.
\end{equation}
The above equation has the form \eqref{linear-matrix-equation} and the matrices $C$, $\sum_{j=1}^kA_jB_j$ are commuting diagonalizable matrices.

Returning to the study of equation \eqref{linear-matrix-equation}, from our hypotheses, the matrices in $\mathcal{N}$ are simultaneously diagonalizable. For $S\in \mathfrak{S}(\mathcal{N})$ and $j\in \{1, \dots, k\}$ we make the following notations:
${\bf a}^j=\Delta_{A_j}(S),\,\,{\bf b}^j=\Delta_{B_j}(S),\,\,{\bf c}=\Delta_{C}(S)$.
The components
 $a_1^{j}, \dots, a_n^{j}$ of ${\bf a}^j$ are the eigenvalues of $A_j$, the components $b_1^{j}, \dots, b_n^{j}$ of ${\bf b}^j$ are the eigenvalues of $B_j$, and the components $c_1, \dots, c_n$ of ${\bf c}$ are the eigenvalues of $C$. 

Making the substitution
$Y=S^{-1}XS$
the equation \eqref{linear-matrix-equation} becomes
\begin{equation}\label{linear-matrix-redus}
\sum_{j=1}^k\text{diag}({\bf a}^j)Y\text{diag}({\bf b}^j)=\text{diag}({\bf c}).
\end{equation}
The above equation is a linear matrix equation of the form \eqref{linear-matrix-equation}. The set $\mathcal{N}_d=\{\text{diag}({\bf a}^1), \dots, \text{diag}({\bf a}^k), \text{diag}({\bf b}^1), \dots, \text{diag}({\bf b}^k), \text{diag}({\bf c})\}$ is a commuting set of diagonal matrices.
Using the components of $Y$ the equation \eqref{linear-matrix-redus} becomes\footnote{We use the $\delta$-Kronecker symbol; $\delta_{rs}=\begin{cases} 1 & \text{if}\,\,r=s \\ 0 & \text{otherwise}\end{cases}$.}
\begin{equation}\label{linear-matrix-redus-components}
\Gamma^S_{rs} y_{rs}=c_r\delta_{rs}, \,\,\,r,s\in \{1,\dots,n\},
\end{equation}
where $\Gamma^S_{rs}:=\sum_{j=1}^k a_r^{j}b_s^{j}$ and they are the components of the relevant matrix of \eqref{linear-matrix-equation} with respect to $S$\footnote{${\bf 1}_{n\times n}\in \mathcal{M}_n(\C)$ is the $n\times n$ all-ones matrix.}: 
\begin{equation}\label{Gamma-def-11}
\Gamma^S=\sum_{j=1}^k{\bf a}^j ({\bf b}^j)^T=\sum_{j=1}^k\text{diag}({\bf a}^j){\bf 1}_{n\times n}\text{diag}({\bf b}^j).
\end{equation}

\begin{lem}
If $V\in \mathfrak{S}(\mathcal{N})$, then there is $\sigma\in \mathcal{S}_n$ such that $\Gamma^V=P_{\sigma}^{-1}\Gamma^SP_{\sigma}$.
\end{lem}

\begin{proof} Let be ${\bf a}^{jV}=\Delta_{A_j}(V)$  and ${\bf b}^{jV}=\Delta_{B_j}(V)$. 
From Theorem \ref{permutation-vectors-221} we have a permutation $\sigma\in \mathcal{S}_n$ such that ${\bf a}^{jV}={\bf a}^j_{\sigma}$ and ${\bf b}^{jV}={\bf b}^j_{\sigma}$, $\forall j\in \{1,\dots, k\}$. Using \eqref{Gamma-def-11} and the  equality $P_{\sigma}{\bf 1}_{n\times n}P_{\sigma}^{-1}={\bf 1}_{n\times n}$ we can write:

\noindent $\Gamma^V  = \sum\limits_{j=1}^k\text{diag}({\bf a}^{jV}){\bf 1}_{n\times n}\text{diag}({\bf b}^{jV}) 
 = \sum\limits_{j=1}^kP_{\sigma}^{-1}\text{diag}({\bf a}^j)P_{\sigma}{\bf 1}_{n\times n}P_{\sigma}^{-1}\text{diag}({\bf b}^j)P_{\sigma}$
$= P_{\sigma}^{-1}\Gamma^SP_{\sigma}.$
\end{proof}
An elementary analysis of the equation \eqref{linear-matrix-redus-components} lead us to the following result in the homogeneous case.

\begin{lem}\label{homogeneous-solutions-44}
We assume that $C=O_n$, $\mathcal{N}$ is a commuting set of diagonalizable matrices, and $S\in \mathfrak{S}(\mathcal{N})$. The following statements hold.
\begin{enumerate}[(i)]
\item The vectorial space
$\mathcal{X}_{h}=S\{Y=[y_{rs}]\in \mathcal{M}_n(\C)\,|\,y_{rs}=0\,\,\text{if}\,\,\Gamma^S_{rs}\neq 0\}S^{-1}\subset  \mathcal{M}_n(\C)$ is the set of solutions of \eqref{linear-matrix-equation} in the homogeneous case.
Its dimension is equal to the number of zeros of the relevant matrix.
\item If all off-diagonal elements of the relevant matrix are non-zero, then $\mathcal{N}\cup \mathcal{X}_h$ is a commuting set of diagonalizable matrices and all matrices in $\mathcal{X}_{h}$ are solutions of the standard linear matrix equation \eqref{standard-equation}.
\end{enumerate}
\end{lem}

\begin{proof}
$(ii)$ We observe that $S\in \mathfrak{S}(\mathcal{N}\cup \mathcal{X}_h)$.
\end{proof}

To study the non-homogeneous case, using the Drazin inverse of a matrix (see Appendix \ref{Moore-Penrose-prop}), we introduce the matrix 
 \begin{equation}\label{X-hat-99}
\widehat{X}:=\Big(\sum_{j=1}^kA_jB_j\Big)^{\bf D}C.
\end{equation} 
\begin{lem}\label{w-X-commute-22}
If $\mathcal{N}$ is a commuting set of diagonalizable matrices, then $\widehat{X}$ is diagonalizable matrix that commutes with all matrices in $\mathcal{N}$.
\end{lem}

\begin{proof}
Using the properties of Drazin inverse (see Appendix \ref{Moore-Penrose-prop}), we have\footnote{If ${\bf a}=\text{vec}(a_1, \dots, a_n)\in \C^n$ and ${\bf b}=\text{vec}(b_1, \dots, b_n)\in \C^n$, then the Hadamard product of ${\bf a}$ and ${\bf b}$ is ${\bf a}\circ {\bf b}=\text{vec}(a_1b_1, \dots, a_nb_n)\in \C^n$.

If $a\in \C$, then
$a^{\boldsymbol{\dagger}}=\begin{cases}
\frac{1}{a} & \text{if}\,\,a\neq 0 \\
0 & \text{if}\,\,a=0.
\end{cases}$  is its generalized inverse.
}:

$\widehat{X} = \big(S\big(\sum\limits_{j=1}^k\text{diag}({\bf a}^j)\text{diag}({\bf b}^j)\big)S^{-1}\big)^{\bf D}S\text{diag}({\bf c})S^{-1}$

\noindent $= S \big(\big(\text{diag}\big(\sum\limits_{j=1}^k{\bf a}^j\circ {\bf b}^j\big)\big)^{\bf D}\text{diag}({\bf c})\big)S^{-1} 
= S \text{diag}\big(\big(\Gamma^S_{11}\big)^{\boldsymbol{\dagger}}c_1, \dots, \big(\Gamma^S_{nn}\big)^{\boldsymbol{\dagger}}c_n\big) S^{-1}.$
\end{proof}

From \eqref{linear-matrix-redus-components} we obtain the following necessary and sufficient conditions for the consistency of the linear matrix equation \eqref{linear-matrix-equation}.
\begin{thm}\label{case-consistent-54}
If $\mathcal{N}$ is a commuting set of diagonalizable matrices, then the following statements are equivalent.
\begin{enumerate}[(i)]
\item The linear matrix equation \eqref{linear-matrix-equation} is consistent. 
\item For all $r\in \{1,\dots, n\}$ we have $\Gamma^S_{rr}\neq 0$ or $c_r=0$. 
 \item $\widehat{X}$ is a solution of the linear matrix equation \eqref{linear-matrix-equation}.
 \item The standard linear matrix equation \eqref{standard-equation} is consistent.
 \item $\widehat{X}$ is a solution of the standard linear matrix equation \eqref{standard-equation}.
\end{enumerate}
\end{thm}

\begin{proof}
An elementary analysis of \eqref{linear-matrix-redus-components} proves the equivalence $(i)\Leftrightarrow (ii)$.
The equivalence $(iii)\Leftrightarrow (v)$ is a consequence of Lemma \ref{w-X-commute-22}. The implications $(iii)\Rightarrow (i)$ and $(v)\Rightarrow (iv)$ are trivial.

$(ii)\Rightarrow (iii)$. From \eqref{linear-matrix-redus-components} we observe that $\widehat{Y}=\text{diag}((\Gamma^S_{11})^{\boldsymbol{\dagger}}c_1,\dots, (\Gamma^S_{nn})^{\boldsymbol{\dagger}}c_n)$
is a solution of \eqref{linear-matrix-redus}. We deduce that $\widehat{X}=S\widehat{Y}S^{-1}$ is a solution of \eqref{linear-matrix-equation}.

$(iv)\Rightarrow (ii)$. The relevant matrix for the standard matrix equation \eqref{standard-equation} has the components $\Gamma_{rs}^{\text{st},\,S}=\sum_{j=1}^k a_r^{j}b_r^{j}$. Because \eqref{standard-equation} is consistent we deduce that for all $r\in \{1,\dots, n\}$ we have $\Gamma^{\text{st},\,S}_{rr}\neq 0$ or $c_r=0$. But, $\Gamma^{\text{st},\,S}_{rr}=\Gamma^{S}_{rr}$.
\end{proof}

\begin{rem}
If the assumptions of the above theorem are not met, then the conclusions are not necessarily true. 
Example \ref{exmp-stein-1} shows a consistent equation \eqref{linear-matrix-equation} where $\mathcal{N}$ is  a commuting set  containing an element which is not a diagonalizable matrix. Example \ref{exmp-stein-2} shows a consistent equation where $\mathcal{N}$ contains only diagonalizable matrices but it is not a commuting set. In both cases the matrix $\widehat{X}$ is not a solution of \eqref{linear-matrix-equation}.
\end{rem}

The next results present some properties about the structure of the solutions of the linear matrix equation \eqref{linear-matrix-equation}.

\begin{thm}\label{uniqueness-general-76}
We assume that $\mathcal{N}$ is a commuting set of diagonalizable matrices and \eqref{linear-matrix-equation} is consistent. The following statements are valid.
\begin{enumerate}[(i)]
\item The solutions of \eqref{linear-matrix-equation} form the affine set $\mathcal{X}$ obtained by translating of the subspace  $\mathcal{X}_{h}\subset\mathcal{M}_n(\C)$ with the matrix $\widehat{X}$ (see Lemma \ref{homogeneous-solutions-44}). 
\item $\widehat{X}$ is the unique solution of \eqref{linear-matrix-equation} if and only if all entries of the relevant matrix are non-zero. 
\item If the standard linear matrix equation \eqref{standard-equation} has an infinite number of solutions, then \eqref{linear-matrix-equation} has an infinite number of solutions.
\item If all off-diagonal elements of the relevant matrix are non-zero, then $\mathcal{N}\cup \mathcal{X}$ is a commuting set of diagonalizable matrices and all matrices in $\mathcal{X}$ are solutions of the standard linear matrix equation \eqref{standard-equation}.
\end{enumerate}
\end{thm}

\begin{proof}
$(iii)$ There is $(r,s)\in \{1,\dots,n\}^2$ such that $\Gamma^{\text{st},S}_{rs}=0$. We have that $\sum_{i=1}^ka^j_rb^j_r=0=\Gamma^S_{rr}$. We apply $(i)$. 

$(iv)$ If $X$ is a solution of \eqref{linear-matrix-equation}, then $X=\widehat{X}+X_h$ where $X_h\in \mathcal{X}_h$. From Lemma \ref{homogeneous-solutions-44} there exists $S\in\mathfrak{S}(\mathcal{N}\cup \mathcal{X}_h)$. Using  Lemma \ref{w-X-commute-22} we obtain that $S\in \mathfrak{S}(\widehat{X})$. Consequently, $S\in \mathfrak{S}(X)$.
\end{proof}

\begin{thm}\label{invertible-general-52}
If $\mathcal{N}$ is a commuting set of diagonalizable matrices, then the following statements hold.
\begin{enumerate}[(i)]
\item If $C\in\mathcal{M}_n^{\text{inv}}(\C)$, then \eqref{linear-matrix-equation} is consistent if and only if $\sum\limits_{j=1}^kA_jB_j\in\mathcal{M}_n^{\text{inv}}(\C)$.
\item If $\sum\limits_{j=1}^kA_jB_j\in\mathcal{M}_n^{\text{inv}}(\C)$, then $\widehat{X}=\big(\sum\limits_{j=1}^kA_jB_j\big)^{-1}C$ is a solution of \eqref{linear-matrix-equation}. In this case, if $X\neq\widehat{X}$ is a solution of \eqref{linear-matrix-equation}, then $\{{A_1}, \dots, {A_k}, X\}$ and $\{{B_1}, \dots,  {B_k}, X\}$ are not commuting sets.
\end{enumerate}
\end{thm}

\begin{proof}
$(i)$ The standard linear matrix equation \eqref{standard-equation} is consistent if and only if  $\sum_{j=1}^kA_jB_j$ is invertible. We use Theorem \ref{case-consistent-54}. 

$(ii)$ The standard linear matrix equation \eqref{standard-equation} has $\widehat{X}:=\big(\sum_{j=1}^kA_jB_j\big)^{-1}C$ as the unique solution. Let $X$ be a solution of \eqref{linear-matrix-equation}. If $\{{A_1}, \dots, {A_k}, X\}$ is a commuting set, then $X=C(\sum_{j=1}^kA_jB_j)^{-1}=\widehat{X}$. 
\end{proof}

\subsection{Linear matrix equations with parameters that form a commuting set of normal matrices}\label{normal-results-22}

In this section we study the linear matrix equation \eqref{linear-matrix-equation} when $\mathcal{N}$ is a commuting set of normal matrices; more precisely, $\forall j,l\in \{1,\dots, k\}$, we have\footnote{If $M\in \mathcal{M}_{m\times n}(\C)$, then $M^{\star}$ is its conjugate transpose.}  $A_jB_l=B_lA_j$, $A_jC=CA_j$, $B_lC=CB_l$, $A_j^{\star}A_j=A_jA_j^{\star}$, $B_l^{\star}B_l=B_lB_l^{\star}$, and $C^{\star}C=CC^{\star}$. 
The set $\mathcal{N}$ is a simultaneously unitarily diagonalizable set (see \cite{horn}, Theorem 2.5.5). 
From Lemma \ref{prop-normal-Moore} we have that $C$ and $\sum_{j=1}^kA_jB_j$ are commuting normal matrices; the parameters of the standard linear matrix equation \eqref{standard-equation} form a commuting set of normal matrices.

All the properties presented in the general case are valid in this case. In what follows, we present some specific properties for the case where $\mathcal{N}$ is a commuting set of normal matrices.
Using the Moore-Penrose inverse (see Appendix \ref{Moore-Penrose-prop}) and Lemma \ref{normal-Drazin-Moore-Penrose}, the matrix defined in \eqref{X-hat-99} can be written
\begin{equation}
\widehat{X}=\Big(\sum_{j=1}^kA_jB_j\Big)^{\bf D}C=\Big(\sum_{j=1}^kA_jB_j\Big)^{\boldsymbol{\dagger}}C.
\end{equation}
$\widehat{X}$ is a normal matrix that commutes with all matrices in $\mathcal{N}$. 

\begin{thm}\label{normal-off-diagonal9}
If $\mathcal{N}$ is a commuting set of normal matrices, then the following statements are equivalent. 
\begin{enumerate}[(i)]
\item All off-diagonal elements of the relevant matrix are non-zero.
\item All solutions of \eqref{linear-matrix-equation} are normal matrices.  
\end{enumerate}
\end{thm}

\begin{proof} Let be\footnote{$\mathcal{U}(n)$ is the set of unitary matrices.} $U\in \mathfrak{S}(\mathcal{N})\cap \mathcal{U}(n)$. 

$(i)\Rightarrow (ii)$. If $X\in \mathcal{X}$, then $X=\widehat{X}+X_h$ with $X_h\in \mathcal{X}_h$. We have $\widehat{X}=U\text{diag}(\widehat{\bf x})U^{\star}$ and, using Lemma \ref{homogeneous-solutions-44}, we can write $X_h=U\text{diag}({\bf x}_h)U^{\star}$.  We deduce that $X$ is a normal matrix.

$(ii)\Rightarrow (i)$. Suppose there exists $(r,s)\in \{1,\dots,n\}^2$, with $r\neq s$, such that $\Gamma^U_{rs}=0$. From Lemma \ref{homogeneous-solutions-44} we have that\footnote{$\{{\bf e}_1,\dots,{\bf e}_n\}$ is the canonical basis in $\C^n$.} $X:=U{\bf e}_r\otimes {\bf e}_s^TU^{\star}$ is an element of $\mathcal{X}_{h}$. Because $XX^{\star}-X^{\star}X=U({\bf e}_r\otimes {\bf e}_s^T-{\bf e}_s\otimes {\bf e}_r^T)U^{\star}\neq O_n$ we deduce that $X$ is not a normal matrix. The obtained contradiction proves the announced result.
\end{proof}

\section{The Sylvester equation and the continuous Lyapunov equation}\label{section-Sylvester}

The Sylvester equation (for square matrices) is a particular case of \eqref{linear-matrix-equation} and it has the form
\begin{equation}\label{Sylvester}
AX+XB=C,
\end{equation}
where $A,B,C\in \mathcal{M}_n(\C)$. 

\begin{thm}\label{Sylvester-thm-88}
If $\{A,B,C\}$ is a commuting set of diagonalizable matrices, then the following statements hold.
\begin{enumerate}[(i)]
\item $\widehat{X}=(A+B)^{\bf D}C$ is a diagonalizable matrix that  commutes with all the matrices in the set $\{A,B,C\}$. If $A,B,C$ are normal matrices, then $\widehat{X}=(A+B)^{\boldsymbol{\dagger}}C$ and this is a normal matrix. 
\item Any of the following conditions is a necessary and sufficient condition for the consistency of \eqref{Sylvester}.
\begin{enumerate}[(a)]
\item $\widehat{X}$ is a solution of the Sylvester equation \eqref{Sylvester}. 
\item The standard linear equation $(A+B)X=C$ is consistent.
\item $\widehat{X}$ is a solution of the standard linear equation $(A+B)X=C$. 
\end{enumerate}
\item If \eqref{Sylvester} is consistent,  $({\bf a}=\emph{vec}(a_1,\dots,a_n), {\bf b}=\emph{vec}(b_1,\dots,b_n))$ is a sequence of induced vectors by $(A,B)$, then 
\begin{enumerate}
\item the dimension of the affine set of solutions of \eqref{Sylvester} is equal to the cardinality of $\{(r,s)\in \{1,\dots,n\}^2\,|\,a_r+b_s=0\}$;
\item $a_r+b_s\neq 0$ when $r\neq s$ is a sufficient condition such that all solutions of \eqref{Sylvester} are diagonalizable matrices.
\end{enumerate}
\item If $A+B$ is invertible, then $\widehat{X}=(A+B)^{-1}C$ is a solution of \eqref{Sylvester}. In addition, if $X\neq \widehat{X}$  is a solution of \eqref{Sylvester}, then $AX\neq XA$ and $BX\neq XB$. 
\item If \eqref{Sylvester} is consistent and $A+B$ is not invertible, then \eqref{Sylvester} has an infinite number of solutions.
\end{enumerate}
\end{thm}

\begin{proof}
For $(i)$ we use Lemma \ref{w-X-commute-22} and the observation from Section \ref{normal-results-22}. The affirmations from $(ii)$ are consequences of Theorem \ref{case-consistent-54}. 

$(iii)$ Let be $S\in \mathfrak{S}(\{A,B,C\})$. $(\widetilde{\bf a}=\Delta_A(S),\widetilde{\bf b}=\Delta_B(S))$ is a sequence of induced vectors by $(A,B)$. From Theorem \ref{permutation-vectors-221} we obtain the existence of $\sigma\in \mathcal{S}_n$ such that ${\bf a}=\widetilde{\bf a}_{\sigma}$ and ${\bf b}=\widetilde{\bf b}_{\sigma}$. From Theorem \ref{description-s(M)-321} we have that $V=SP_{\sigma}\in \mathfrak{S}(\{A,B,C\})$. The equality $\Gamma^V_{rs}=a_r+b_s$ is obtained from \eqref{Gamma-def-11}. We apply Theorem \ref{uniqueness-general-76} and Lemma \ref{homogeneous-solutions-44}. 

To prove $(iv)$ and $(v)$ we use Theorem \ref{invertible-general-52}.
\end{proof}

\begin{rem}
The sequence of induced vectors $({\bf a}, {\bf b})$ can be obtained if we know a matrix $S\in \mathfrak{S}(\{A,B,C\})$ or it can be obtained from the eigenvalues of $A$ and $B$ by applying the algorithm presented in Section \ref{Gamma-without-S}. 
\end{rem}\medskip

A particular case of Sylvester equation is the continuous Lyapunov equation
\begin{equation}\label{continuous-Lyapunov}
A^*X+XA=C,
\end{equation}
with $C$ a Hermitian matrix ($C=C^{\star}$).
The matrices $A$ and $A^{\star}$ commute if and only if $A$ and $A^{\star}$ are normal matrices. $C$ is a normal matrix because it is Hermitian matrix. 

\begin{thm}\label{continuous-Lyapunov-thm-67}
Let be $A$ a normal
 matrix having the eigenvalues $a_1$, ..., $a_n$ and $C$ a Hermitian matrix such that $AC=CA$. The following statements hold.
\begin{enumerate}[(i)]
\item $\widehat{X}=(A+A^{\star})^{\bf D}C=(A+A^{\star})^{\boldsymbol{\boldsymbol{\dagger}}}C$ is a Hermitian matrix that commutes with all the matrices in the set $\{A, A^{\star}, C\}$.
\item Any of the following conditions is a necessary and sufficient condition for the consistency of the continuous Lyapunov equation \eqref{continuous-Lyapunov}.
\begin{enumerate}[(a)]
\item $\widehat{X}$ is a solution of \eqref{continuous-Lyapunov}. 
\item The standard linear equation $(A+A^{\star})X=C$ is consistent.
\item $\widehat{X}$ is a solution of the standard linear equation $(A+A^{\star})X=C$. 
\end{enumerate}
\item If \eqref{continuous-Lyapunov} is consistent, then the dimension of the affine set of the solutions is equal to the cardinality of $\{(r,s)\in \{1,\dots, n\}^2\,|\,\overline{a}_r+a_s= 0\}$. 
\item If \eqref{continuous-Lyapunov} is consistent, then the following statements are equivalent.
\begin{enumerate}[(a)]
\item All solutions of \eqref{continuous-Lyapunov} are normal matrices.
\item If $r,s\in \{1,\dots,n\}$ and $r\neq s$, then $\overline{a}_r+a_s\neq 0$.
\end{enumerate}
\item If $A+A^{\star}$ is invertible, then $\widehat{X}=(A+A^{\star})^{-1}C$ is a solution of \eqref{continuous-Lyapunov}. In addition, if $X\neq \widehat{X}$  is a solution of \eqref{continuous-Lyapunov}, then $AX\neq XA$, $A^{\star}X\neq XA^{\star}$. 

\item If \eqref{continuous-Lyapunov} is consistent and $A+A^{\star}$ is not invertible, then \eqref{continuous-Lyapunov} has an infinite number of solutions.
\end{enumerate}
\end{thm}

\begin{proof} 
Because $A^{\star}C=A^{\star}C^{\star}=(CA)^{\star}=(AC)^{\star}=C^{\star}A^{\star}=CA^{\star}$ we obtain that $\{A,A^{\star},C\}$ is a commuting set of normal matrices.

$(i)$ We prove that $\widehat{X}$ is a Hermitian matrix.
Using the properties of the Moore-Penrose inverse (Appendix \ref{Moore-Penrose-prop}) we have
\begin{align*}
\widehat{X}^{\star} & =((A+A^{\star})^{\boldsymbol{\dagger}}C)^{\star}=C^{\star}((A+A^{\star})^{\boldsymbol{\dagger}})^{\star}=C(A+A^{\star})^{\boldsymbol{\dagger}}=(A+A^{\star})^{\boldsymbol{\dagger}}C=\widehat{X}.
\end{align*}
$(iii)$ If ${\bf a}=\text{vec}(a_1,\dots,a_n)$, then $\overline{\bf a}=\text{vec}(\overline{a_1},\dots,\overline{a_n})$ and $({\bf a},\overline{\bf a})$ is a sequence of induced vectors by $(A,A^{\star})$.

For the other announced results, we apply Theorem \ref{Sylvester-thm-88} and Theorem \ref{normal-off-diagonal9}.
\end{proof}

\begin{rem}
Example \ref{continuous-Lyapinov-not-normal-23} presents a homogeneous continuous Lyapunov equation with a parameter $A$, that is not a normal matrix, for which the result given in Theorem \ref{continuous-Lyapunov}-$(iii)$ is not true. 
\end{rem}

\section{The Stein equation and the discrete Lyapunov equation}\label{section-Stein}

Another important particular case of \eqref{linear-matrix-equation} is the Stein equation:
\begin{equation}\label{Stein}
AXB-X=C,
\end{equation}
with $A,B,C\in \mathcal{M}_n(\mathbb{C})$.
Similar arguments to those in the proof of Theorem \ref{Sylvester-thm-88} lead us to the following results.
\begin{thm}\label{Stein-thm-878}
If $\{A,B,C\}$ is a commuting set of diagonalizable matrices, then the following statements hold.
\begin{enumerate}[(i)]
\item $\widehat{X}=(AB-I_n)^{\bf D}C$ is a diagonalizable matrix that  commutes with all the matrices in the set $\{A,B,C\}$. If $A,B,C$ are normal matrices, then $\widehat{X}=(AB-I_n)^{\boldsymbol{\dagger}}C$ and this is a normal matrix.
\item Any of the following conditions is a necessary and sufficient condition for the consistency of the Stein equation \eqref{Stein}.
\begin{enumerate}[(a)]
\item $\widehat{X}$ is a solution of the Stein equation \eqref{Stein}. 
\item The standard linear equation $(AB-I_n)X=C$ is consistent.
\item $\widehat{X}$ is a solution of the standard linear equation $(AB-I_n)X=C$. 
\end{enumerate}
\item If \eqref{Stein} is consistent and $({\bf a}=\emph{vec}(a_1,\dots,a_n), {\bf b}=\emph{vec}(b_1,\dots,b_n))$ is a sequence of induced vectors by $(A,B)$, then 
\begin{enumerate}
\item the dimension of the affine set of solutions of \eqref{Stein} is the cardinality of $\{(r,s)\in \{1,\dots,n\}^2\,|\,a_rb_s=1\}$;
\item $a_rb_s\neq 1$ when $r\neq s$ is a sufficient condition such that all solutions of \eqref{Stein} are diagonalizable matrices.
\end{enumerate}
\item If $AB-I_n$ is invertible, then $\widehat{X}=(AB-I_n)^{-1}C$. In addition, if $X\neq \widehat{X}$  is a solution of \eqref{Stein}, then $AX\neq XA$ and $BX\neq XB$. 
\item If \eqref{Stein} is consistent and $AB-I_n$ is not invertible, then \eqref{Stein} has an infinite number of solutions.
\end{enumerate}
\end{thm}

\begin{rem}
The sequence of induced vectors $({\bf a}, {\bf b})$ can be obtained applying the algorithm presented in Section \ref{Gamma-without-S}. 
\end{rem}\medskip

A special case of Stein equation is the discrete Lyapunov equation:  
\begin{equation}\label{discrete-Lyapunov}
A^*XA-X=C
\end{equation}
with $C$ a Hermitian matrix ($C=C^{\star}$).
The matrices $A$ and $A^{\star}$ commute if and only if $A$ and $A^{\star}$ are normal matrices. $C$ is a normal matrix because it is Hermitian matrix. 

\begin{thm}
Let be $A$ a normal
 matrix having the eigenvalues $a_1$, ..., $a_n$ and $C$ a Hermitian matrix such that $AC=CA$. The following statements hold.
\begin{enumerate}[(i)]
\item $\widehat{X}=(AA^{\star}-I_n)^{\boldsymbol{\boldsymbol{\dagger}}}C$ is a Hermitian matrix that commutes with all the matrices in $\{A, A^{\star}, C\}$.
\item Any of the following conditions is a necessary and sufficient condition for the consistency of the discrete Lyapunov equation \eqref{discrete-Lyapunov}.
\begin{enumerate}[(a)]
\item $\widehat{X}$ is a solution of \eqref{discrete-Lyapunov}. 
\item The standard linear equation $(AA^{\star}-I_n)X=C$ is consistent.
\item $\widehat{X}$ is a solution of the standard linear equation $(AA^{\star}-I_n)X=C$. 
\end{enumerate}
\item If \eqref{discrete-Lyapunov} is consistent, then the dimension of the affine set of the solutions is equal to the cardinality of $\{(r,s)\in \{1,\dots, n\}^2\,|\,\overline{a}_r a_s= 1\}$. 
\item If \eqref{discrete-Lyapunov} is consistent, then the following statements are equivalent.
\begin{enumerate}[(a)]
\item All solutions of \eqref{discrete-Lyapunov} are normal matrices.
\item If $r,s\in \{1,\dots,n\}$ and $r\neq s$, then $\overline{a}_ra_s\neq 1$.
\end{enumerate}
\item If $AA^{\star}-I_n$ is invertible, then $\widehat{X}=(AA^{\star}-I_n)^{-1}C$. In addition, if $X\neq \widehat{X}$  is a solution of \eqref{continuous-Lyapunov}, then $AX\neq XA$, $A^{\star}X\neq XA^{\star}$. 
\end{enumerate} 
\end{thm}

\begin{proof}  
Analogous to the proof from Theorem \ref{continuous-Lyapunov-thm-67} we obtain that $\{A,A^{\star},C\}$ is a commuting set of normal matrices.
Using the properties of the Moore-Penrose inverse (see Appendix \ref{Moore-Penrose-prop}) we have
\begin{align*}
\widehat{X}^{\star} & =((AA^{\star}-I_n)^{\boldsymbol{\dagger}}C)^{\star}=C^{\star}((AA^{\star}-I_n)^{\boldsymbol{\dagger}})^{\star}=C(AA^{\star}-I_n)^{\boldsymbol{\dagger}} \\
& =(AA^{\star}-I_n)^{\boldsymbol{\dagger}}C=\widehat{X}.
\end{align*}
\
For the other announced results, we apply Theorem \ref{Stein-thm-878} and Theorem \ref{normal-off-diagonal9}.
\end{proof}

\begin{rem}
Example \ref{discrete-Lyapunov-not-normal-45} presents a homogeneous discrete Lyapunov equation with a parameter $A$, that is not a normal matrix, for which the result given in Theorem \ref{discrete-Lyapunov}-$(iii)$ is not true. 
\end{rem}

\section{The diagonalizing matrices for a commuting set of diagonalizable matrices}\label{diagonalizing-matrices-654}

First, we present a study of the matrices that commute with a  diagonalizable matrix.
If $M\in \mathcal{M}_n(\C)$, then the set $\mathfrak{C}(M)$ of the matrices that commute with $M$ are the solutions of the homogeneous linear matrix equation
\begin{equation}\label{commuting-set-A-22}
MX-XM=O_n.
\end{equation}
This equation is a particular case of the Sylvester equation \eqref{Sylvester}. 

We assume that $M$ is a diagonalizable matrix and we note by $\lambda_1, \dots, \lambda_d$ the distinct eigenvalues of $M$ and $k_1, \dots, k_d$ their algebraic multiplicity.
From Lemma \ref{U-P-234} we have that ${\bf m}=\text{vec}(\lambda_1 {\bf 1}_{k_1}, \dots, \lambda_d {\bf 1}_{k_d})\in \C^n$ is a star vector induced by $M$.
We choose $S\in \mathfrak{S}(M)$ such that $\Delta_M(S)={\bf m}$.
The relevant matrix with respect to $S$ is:\footnote{${\bf 1}_{p\times q}\in \mathcal{M}_{p\times q}(\C)$ is the all-ones matrix.} 
$$\Gamma^S=\begin{pmatrix}
O_{k_1} & (\lambda_1-\lambda_2){\bf 1}_{k_1\times k_2} & \dots & (\lambda_1-\lambda_d){\bf 1}_{k_1\times k_d}\\
(\lambda_2-\lambda_1){\bf 1}_{k_2\times k_1} & O_{k_2} & \dots & (\lambda_2-\lambda_d){\bf 1}_{k_2\times k_d}\\
\vdots & \vdots & \ddots & \vdots \\
(\lambda_d-\lambda_1){\bf 1}_{k_d\times k_1} &  (\lambda_d-\lambda_2){\bf 1}_{k_d\times k_2} & \dots & O_{k_d} 
\end{pmatrix}.$$

From Lemma \ref{homogeneous-solutions-44} we obtain the following result.

\begin{thm}\label{commutator-A-22}
If the above hypotheses are satisfied, then\footnote{The direct sum of the matrices $Y_1\in \mathcal{M}_{k_1}(\C)$,  $Y_2\in \mathcal{M}_{k_2}(\C)$, ...,  $Y_d\in \mathcal{M}_{k_d}(\C)$ is
\begin{equation}\label{direct-sum-90}
Y_1\oplus Y_2 \oplus \dots \oplus Y_d:=\begin{pmatrix} Y_1 & O_{k_1\times k_2} & \dots & O_{k_1\times k_d} \\
 O_{k_2\times k_1} & Y_2 & \dots &  O_{k_2\times k_d} \\
\vdots & \vdots & \ddots & \vdots \\
 O_{k_d\times k_1} &  O_{k_d\times k_2} & \dots & Y_d
\end{pmatrix}
\in \mathcal{M}_n(\C),
\end{equation}
where $O_{p\times q}\in \mathcal{M}_{p\times q}(\C)$ is the zero matrix and $n=k_1+\dots +k_d$.}  
 $$\mathfrak{C}(M)=\{S (Y_1 \oplus \dots \oplus Y_d)S^{-1}\,|\,Y_1\in \mathcal{M}_{k_1}(\C), ..., Y_d\in \mathcal{M}_{k_d}(\C)\}$$ and it is a vectorial subspace of $\mathcal{M}_n(\C)$ with the dimension $k_1^2+k_2^2+\dots+ k_d^2\geq n$.
\end{thm}

\begin{cor}\label{corolar-1234}
If ${\bf m}=\emph{vec}(\lambda_1 {\bf 1}_{k_1}, \dots, \lambda_d {\bf 1}_{k_d})\in \C^n$ with $\lambda_1, \dots, \lambda_d$ distinct numbers, then
$\mathfrak{C}(\emph{diag}({\bf m}))=\{Y_1 \oplus \dots \oplus Y_d\,|\,Y_1\in \mathcal{M}_{k_1}(\C), ..., Y_d\in \mathcal{M}_{k_d}(\C)\}.$
\end{cor}

The above results allow detailing the structure of the set $\mathfrak{S}(M)$ (we recover Theorem 1.3.27 from \cite{horn}).
\begin{lem}\label{horn-1-3-27} Let be $M\in \mathcal{M}_n(\C)$ a diagonalizable matrix. If $\lambda_1, \dots, \lambda_d$ are the distinct eigenvalues of $M$, $k_1, \dots, k_d$ are their algebraic multiplicity, and $S\in \mathfrak{S}(M)$ such that $\Delta_M(S)={\bf m}=\emph{vec}(\lambda_1 {\bf 1}_{k_1}, \dots, \lambda_d {\bf 1}_{k_d})\in \C^n$, then the following statements hold.
\begin{enumerate}[(i)]
\item $V\in \mathfrak{S}(M)$ and $\Delta_M(V)={\bf m}$ if and only if $V=S (Y_1 \oplus \dots \oplus Y_d)$ in which each $Y_j\in \mathcal{M}_{k_j}(\C)$ is nonsingular. 

\item $\mathfrak{S}(M)=\{S (Y_1 \oplus \dots \oplus Y_d)P_{\sigma}\,|\,\sigma\in \mathcal{S}_n,\,Y_j\in \mathcal{M}^{\emph{inv}}_{k_j}(\C)\}$. Also, we have $\Delta_M(S (Y_1 \oplus \dots \oplus Y_d)P_{\sigma})={\bf m}_{\sigma}$.
\end{enumerate}
\end{lem}

\begin{proof}
$(i)$ We assume that $V\in \mathfrak{S}(M)$ and $\Delta_M(V)={\bf m}$.
Because $S^{-1}MS=V^{-1}MV$ we obtain that $VS^{-1}\in \mathfrak{C}(M)$. From Theorem \ref{commutator-A-22} we deduce that $V=S (Y_1 \oplus \dots \oplus Y_d)$ with $Y_1\in \mathcal{M}_{k_1}(\C)$, ..., and $Y_d\in \mathcal{M}_{k_d}(\C)$. The fact that $V$ is invertible implies that $Y_1$, ..., $Y_d$ are invertible matrices.  

If $V=S (Y_1 \oplus \dots \oplus Y_d)$ with $Y_1\in \mathcal{M}^{\text{inv}}_{k_1}(\C)$, ..., and $Y_d\in \mathcal{M}^{\text{inv}}_{k_d}(\C)$, then
\begin{align*}
V^{-1}MV & = (Y_1^{-1} \oplus \dots \oplus Y_d^{-1})S^{-1}MS(Y_1 \oplus \dots \oplus Y_d) \\
 & =(Y_1^{-1} \oplus \dots \oplus Y_d^{-1})\text{diag}({\bf m})(Y_1 \oplus \dots \oplus Y_d) \\
& =(Y_1^{-1} \oplus \dots \oplus Y_d^{-1})(\lambda_1I_{k_1}\oplus \dots \oplus \lambda_dI_{k_d})(Y_1 \oplus \dots \oplus Y_d) \\
& =\lambda_1I_{k_1}\oplus \dots \oplus \lambda_dI_{k_d}=\text{diag}({\bf m}).
\end{align*}
$(ii)$ Using $(i)$ and Lemma \ref{U-P-234} we obtain that the matrix $S (Y_1 \oplus \dots \oplus Y_d)P_{\sigma}$, with  $Y_j\in \mathcal{M}^{\text{inv}}_{k_j}(\C)$, is an element of $\mathfrak{S}(M)$.
If $W\in \mathfrak{S}(M)$, then, from Lemma \ref{U-P-234}, there is $\sigma\in \mathcal{S}_n$ such that $\Delta_M(W)={\bf m}_{\sigma}$. We deduce that $\Delta_M(WP_{\sigma^{-1}})={\bf m}$. Consequently, $W=S (Y_1 \oplus \dots \oplus Y_d)P_{\sigma}$ with $Y_j\in \mathcal{M}^{\text{inv}}_{k_j}(\C)$.
\end{proof}

\begin{rem}
If the distinct complex numbers $\lambda_1, \dots, \lambda_n$ are the eigenvalues of $M\in \mathcal{M}_n(\C)$, then $M$ is a diagonalizable matrix (see \cite{horn}, Theorem 1.3.9). There is $S\in \mathfrak{S}(M)$ such that $\Delta_M(S)={\bf m}:=\text{vec}(\lambda_1, \dots, \lambda_n)$. 
$V\in \mathfrak{S}(M)$ with $\Delta_M(V)={\bf m}$  
if and only if $V=S\text{diag}({\bf y})$ with ${\bf y}\in (\C^*)^n$. 
Also, we have
$\mathfrak{S}(M)=\{S\text{diag}({\bf y})P_{\sigma}\,|\,{\bf y}\in (\C^*)^n, \sigma\in \mathcal{S}_n\}$. 
\end{rem}

In what follows we present a generalization of the above result.

\begin{defn} 
$(I_1, \dots, I_q)\subset \mathbb{N}^*\times \dots \times (\mathbb{N}^*)^q$ is a star sequence of indices if $I_1=\{\alpha\in \mathbb{N}^*\,|\,\alpha\leq \mathfrak{c}(1)\}, \,\mathfrak{c}(1)\in \mathbb{N}^*$, and for $j\in \{1,\dots,q-1\}$, $\boldsymbol{\alpha}\in I_j$ there is $\delta_{\boldsymbol{\alpha}}\in \mathbb{N}^*$ such that $I_{j+1}=\{(\boldsymbol{\alpha}, i)\,|\,\boldsymbol{\alpha}\in I_j\,,1\leq i\leq \delta_{\boldsymbol{\alpha}}\}$. 

We note by $\mathfrak{c}(j)$ the number of elements of $I_j$ and, using the lexicographic order on $\mathbb{N}^j$, we have $I_j=\{\boldsymbol{\alpha}^j_1, \dots, \boldsymbol{\alpha}^j_{\mathfrak{c}(j)}\}$ with $\boldsymbol{\alpha}^j_1< \dots< \boldsymbol{\alpha}^j_{\mathfrak{c}(j)}$. 
\end{defn}

\begin{defn}\label{definition-principal-vector88}
$({\bf m}^1,\dots, {\bf m}^q)$ is a star sequence of vectors from $\C^n$ if there is a star sequence of indices $(I_1, \dots, I_q)$ such that we have:

$\bullet\,\,\,{\bf m}^1=\emph{vec}(\lambda_1 {\bf 1}_{k_1}, \dots, \lambda_{\mathfrak{c}(1)} {\bf 1}_{k_{\mathfrak{c}(1)}})$, where $\lambda_1, \dots, \lambda_{\mathfrak{c}(1)}$ are different complex numbers and $k_1+\dots+k_{\mathfrak{c}(1)}=n$;

$\bullet\,\,\,{\bf m}^j=\emph{vec}(\lambda_{\boldsymbol{\alpha}^j_1} {\bf 1}_{k_{\boldsymbol{\alpha}^j_1}}, \dots,  \lambda_{\boldsymbol{\alpha}^j_{\mathfrak{c}(j)}} {\bf 1}_{k_{\boldsymbol{\alpha}^j_{\mathfrak{c}(j)}}}),\,j>1,$
where for $\boldsymbol{\beta}\in I_{j-1}$, the complex numbers $\lambda_{(\boldsymbol{\beta},1)}$, ..., $\lambda_{(\boldsymbol{\beta},\delta_{\beta})}$ are different and $\sum_{i=1}^{\delta_{\boldsymbol{\beta}}}k_{(\boldsymbol{\beta},i)}=k_{\boldsymbol{\beta}}$. 
\end{defn}
A star sequence of vectors has the following properties:
\begin{enumerate}[(i)]
 \item ${\bf m}^1$  is a star vector;
\item $\sum_{\boldsymbol{\beta}\in I_{j}}k_{\boldsymbol{\beta}}=n$;
\item ${\bf m}^j=({\bf m}^{j,\boldsymbol{\alpha}^{j-1}_1}, \dots, {\bf m}^{j,\boldsymbol{\alpha}^{j-1}_{\mathfrak{c}(j-1)}})$, where,  for all $\boldsymbol{\beta}\in I_{j-1}$,  the vector
${\bf m}^{j,\boldsymbol{\beta}}=\text{vec}(\lambda_{(\boldsymbol{\beta},1)} {\bf 1}_{k_{(\boldsymbol{\beta},1)}}, \dots, \lambda_{(\boldsymbol{\beta},\delta_{\beta})} {\bf 1}_{k_{(\boldsymbol{\beta},\delta_{\beta})}})$
is a star vector in $\C^{k_{\boldsymbol{\beta}}}$. 
\end{enumerate}

In the following theorem we use the above notations.

\begin{thm}\label{commute-set-1001}
If $\mathcal{M}=(M_1,\dots,M_q)$ is a sequence of commuting diagonalizable matrices, $({\bf m}^1=\Delta_{M_1}(S), \dots, {\bf m}^q=\Delta_{M_q}(S))$ is a star sequence of induced vectors, $S\in \mathfrak{S}(\mathcal{M})$, then the following statements hold.

\begin{enumerate}[(i)]
\item $V\in \mathcal{M}^\emph{inv}_n(\C)$ satisfies $V\in \mathfrak{S}(\mathcal{M})$, ${\bf m}^1=\Delta_{M_1}(V)$, ..., and ${\bf m}^q=\Delta_{M_q}(V)$ if and only if $V=S(Y_{\boldsymbol{\alpha}^q_1}\oplus \dots \oplus Y_{\boldsymbol{\alpha}^q_{\mathfrak{c}(q)}})$, where $Y_{\boldsymbol{\alpha}^q_p} \in \mathcal{M}^{\emph{inv}}_{k_{{\boldsymbol{\alpha}^q_p}}}(\C)$. 
\item If $\sigma_1\in \mathcal{S}_{k_{\boldsymbol{\alpha}^q_1}}$, ..., $\sigma_{\mathfrak{c}(q)}\in \mathcal{S}_{k_{\boldsymbol{\alpha}^q_{\mathfrak{c}(q)}}}$ and $\sigma:=\sigma_1\oplus\dots\oplus \sigma_{\mathfrak{c}(q)}\in \mathcal{S}_n$ (Appendix \ref{permutation-matrix-22}), then $\Delta_{M_1}(SP_{\sigma})=\Delta_{M_1}(S)$, ..., $\Delta_{M_q}(SP_{\sigma})=\Delta_{M_q}(S)$. 
\item Let be $\mathcal{N}=(M_1,\dots,M_q, N)$ a sequence of commuting diagonalizable matrices, $\Delta_N(S)=\emph{vec}(\widehat{\bf n}^{\boldsymbol{\alpha}^q_1}, \dots, \widehat{\bf n}^{\boldsymbol{\alpha}^q_{\mathfrak{c}(q)}})\in \C^{k_{\boldsymbol{\alpha}^q_1}}\times \dots\times \C^{k_{\boldsymbol{\alpha}^q_{\mathfrak{c}(q)}}}$, $({\bf m}^1, \dots, {\bf m}^q, {\bf n})$ is a sequence of induced vectors by $\mathcal{N}$ with ${\bf n}=\emph{vec}({\bf n}^{\boldsymbol{\alpha}^q_1}, \dots, {\bf n}^{\boldsymbol{\alpha}^q_{\mathfrak{c}(q)}})\in \C^{k_{\boldsymbol{\alpha}^q_1}}\times \dots\times \C^{k_{\boldsymbol{\alpha}^q_{\mathfrak{c}(q)}}}$, then the vectors   
 ${\bf n}^{\boldsymbol{\alpha}^q_i}$ and $\widehat{\bf n}^{\boldsymbol{\alpha}^q_i}$, $i\in \{1,\dots, \mathfrak{c}(q)\}$, have the same components (possibly in a different order).
\end{enumerate}
\end{thm}

\begin{proof}
$(i)$ We assume that $V=S(Y_{\boldsymbol{\alpha}^q_1}\oplus \dots \oplus Y_{\boldsymbol{\alpha}^q_{\mathfrak{c}(q)}})$, where $Y_{\boldsymbol{\alpha}^q_p} \in \mathcal{M}^{\text{inv}}_{k_{{\boldsymbol{\alpha}^q_p}}}(\C)$.
For $r\in \{1,\dots,q\}$, because $({\bf m}^1, \dots, {\bf m}^q)$ is a star sequence of vectors, we can write $Y_{\boldsymbol{\alpha}^q_1}\oplus \dots \oplus Y_{\boldsymbol{\alpha}^q_{\mathfrak{c}(q)}}=Y_{\boldsymbol{\alpha}^r_1}\oplus \dots \oplus Y_{\boldsymbol{\alpha}^r_{\mathfrak{c}(r)}}$ with  $Y_{\boldsymbol{\alpha}^r_p} \in \mathcal{M}^{\text{inv}}_{k_{{\boldsymbol{\alpha}^r_p}}}(\C)$. We have
\begin{align*}
&V^{-1}M_{r}V
 = (Y^{-1}_{\boldsymbol{\alpha}^r_1}\oplus \dots \oplus Y^{-1}_{\boldsymbol{\alpha}^r_{\mathfrak{c}(r)}})S^{-1}M_{r}S(Y_{\boldsymbol{\alpha}^r_1}\oplus \dots \oplus Y_{\boldsymbol{\alpha}^r_{\mathfrak{c}(r)}}) \\
& =(Y^{-1}_{\boldsymbol{\alpha}^r_1}\oplus \dots \oplus Y^{-1}_{\boldsymbol{\alpha}^r_{\mathfrak{c}(r)}})\text{diag}({\bf m}^r)(Y_{\boldsymbol{\alpha}^r_1}\oplus \dots \oplus Y_{\boldsymbol{\alpha}^r_{\mathfrak{c}(r)}})\\
& = (Y^{-1}_{\boldsymbol{\alpha}^r_1}\oplus \dots \oplus Y^{-1}_{\boldsymbol{\alpha}^r_{\mathfrak{c}(r)}})
(\lambda_{\boldsymbol{\alpha}^r_1}I_{k_{{\boldsymbol{\alpha}^r_1}}}\oplus \dots \oplus \lambda_{\boldsymbol{\alpha}^r_{\mathfrak{c}(r)}}I_{k_{{\boldsymbol{\alpha}^r_{\mathfrak{c}(r)}}}})
(Y_{\boldsymbol{\alpha}^r_1}\oplus \dots \oplus Y_{\boldsymbol{\alpha}^r_{\mathfrak{c}(r)}})\\
& = (\lambda_{\boldsymbol{\alpha}^r_1}Y^{-1}_{\boldsymbol{\alpha}^r_1}Y_{\boldsymbol{\alpha}^r_1})\oplus \dots \oplus (\lambda_{\boldsymbol{\alpha}^r_{\mathfrak{c}(r)}}Y^{-1}_{\boldsymbol{\alpha}^r_{\mathfrak{c}(r)}}Y_{\boldsymbol{\alpha}^r_{\mathfrak{c}(r)}}) = \text{diag}({\bf m}^r).
\end{align*}

We assume that $V$ satisfies $V^{-1}M_jV=\text{diag}({\bf m}^j)$, for all $j\in \{1,\dots,q\}$. We prove by induction that for all $r\in \{1,\dots,q\}$ we have $V=S(Y_{\boldsymbol{\alpha}^r_1}\oplus \dots \oplus Y_{\boldsymbol{\alpha}^r_{\mathfrak{c}(r)}})$, where $Y_{\boldsymbol{\alpha}^r_p} \in \mathcal{M}^{\text{inv}}_{k_{{\boldsymbol{\alpha}^r_p}}}(\C)$. For $r=1$ the announced result is proved in Lemma \ref{horn-1-3-27}. We assume that we have $V=S(Y_{\boldsymbol{\alpha}^r_1}\oplus \dots \oplus Y_{\boldsymbol{\alpha}^r_{\mathfrak{c}(r)}})$ and we obtain 
\begin{align*}
\text{diag}({\bf m}^{r+1}) & =V^{-1}M_{r+1}V  \\
& = (Y^{-1}_{\boldsymbol{\alpha}^r_1}\oplus \dots \oplus Y^{-1}_{\boldsymbol{\alpha}^r_{\mathfrak{c}(r)}})S^{-1}M_{r+1}S(Y_{\boldsymbol{\alpha}^r_1}\oplus \dots \oplus Y_{\boldsymbol{\alpha}^r_{\mathfrak{c}(r)}}) \\
& =(Y^{-1}_{\boldsymbol{\alpha}^r_1}\oplus \dots \oplus Y^{-1}_{\boldsymbol{\alpha}^r_{\mathfrak{c}(r)}})\text{diag}({\bf m}^{r+1})(Y_{\boldsymbol{\alpha}^r_1}\oplus \dots \oplus Y_{\boldsymbol{\alpha}^r_{\mathfrak{c}(r)}}).
\end{align*}
We write
$\text{diag}({\bf m}^{r+1})=\text{diag}({\bf m}^{r+1,\boldsymbol{\alpha}^r_1})\oplus \dots \oplus \text{diag}({\bf m}^{r+1,\boldsymbol{\alpha}^r_{\mathfrak{c}(r)}}) $ and we obtain that  $Y_{\boldsymbol{\alpha}^r_j}\in \mathfrak{C}(\text{diag}({\bf m}^{r+1,\boldsymbol{\alpha}^r_j}))$, $j\in \{1, \dots, \mathfrak{c}(r)\}$.
Using Corollary \ref{corolar-1234} and the fact that
${\bf m}^{r+1,\boldsymbol{\alpha}^r_j}=\text{vec}(\lambda_{(\boldsymbol{\alpha}^r_j,1)} {\bf 1}_{k_{(\boldsymbol{\alpha}^r_j,1)}}, \dots, \lambda_{(\boldsymbol{\alpha}^r_j,\delta_{\boldsymbol{\alpha}^r_j})} {\bf 1}_{k_{(\boldsymbol{\alpha}^r_j,\delta_{\boldsymbol{\alpha}^r_j})}})$ we deduce that $Y_{\boldsymbol{\alpha}^r_j}=Y_{(\boldsymbol{\alpha}^r_j,1)}\oplus \dots \oplus Y_{(\boldsymbol{\alpha}^r_j,\delta_{\boldsymbol{\alpha}^r_j})}$. Consequently, $V=S(Y_{\boldsymbol{\alpha}^{r+1}_1}\oplus \dots \oplus Y_{\boldsymbol{\alpha}^{r+1}_{\mathfrak{c}(r+1)}})$.

$(ii)$ Because $P_{\sigma}=P_{\sigma_1}\oplus \dots \oplus P_{\sigma_{\mathfrak{c}(q)}}$, with $P_{\sigma_p} \in \mathcal{M}^{\text{inv}}_{k_{{\boldsymbol{\alpha}^q_p}}}(\C)$ we can apply $(i)$.

$(iii)$ There is $V\in \mathfrak{S}(\mathcal{N})$ such that ${\bf m}^j= \Delta_{M_j}(V)$, $\forall j\in\{1,\dots, q\}$, and ${\bf n}=\Delta_{N}(V)$.  From $(i)$ we have $V=S(Y_{\boldsymbol{\alpha}^q_1}\oplus \dots \oplus Y_{\boldsymbol{\alpha}^q_{\mathfrak{c}(q)}})$, where $Y_{\boldsymbol{\alpha}^q_p} \in \mathcal{M}^{\text{inv}}_{k_{{\boldsymbol{\alpha}^q_p}}}(\C)$.
\begin{align*}
& \text{diag}({\bf n}) =  \text{diag}({\bf n}^{\boldsymbol{\alpha}^q_1})\oplus \dots \oplus \text{diag}({\bf n}^{\boldsymbol{\alpha}^q_{\mathfrak{c}(q)}}) \\
& = (Y^{-1}_{\boldsymbol{\alpha}^q_1}\oplus \dots \oplus Y^{-1}_{\boldsymbol{\alpha}^q_{\mathfrak{c}(q)}})S^{-1}NS(Y_{\boldsymbol{\alpha}^q_1}\oplus \dots \oplus Y_{\boldsymbol{\alpha}^q_{\mathfrak{c}(q)}}) \\
& = (Y^{-1}_{\boldsymbol{\alpha}^q_1}\oplus \dots \oplus Y^{-1}_{\boldsymbol{\alpha}^q_{\mathfrak{c}(q)}})(\text{diag}(\widehat{\bf n}^{\boldsymbol{\alpha}^q_1})\oplus \dots \oplus \text{diag}(\widehat{\bf n}^{\boldsymbol{\alpha}^q_{\mathfrak{c}(q)}}))(Y_{\boldsymbol{\alpha}^q_1}\oplus \dots \oplus Y_{\boldsymbol{\alpha}^q_{\mathfrak{c}(q)}})\\
& = Y^{-1}_{\boldsymbol{\alpha}^q_1}\text{diag}(\widehat{\bf n}^{\boldsymbol{\alpha}^q_1})Y_{\boldsymbol{\alpha}^q_1}\oplus \dots\oplus Y^{-1}_{\boldsymbol{\alpha}^q_{\mathfrak{c}(q)}}\text{diag}(\widehat{\bf n}^{\boldsymbol{\alpha}^q_{\mathfrak{c}(q)}})Y_{\boldsymbol{\alpha}^q_{\mathfrak{c}(q)}}.
\end{align*}
Consequently, $\text{diag}({\bf n}^{\boldsymbol{\alpha}^q_i})=Y^{-1}_{\boldsymbol{\alpha}^q_i}\text{diag}(\widehat{\bf n}^{\boldsymbol{\alpha}^q_i})Y_{\boldsymbol{\alpha}^q_i}$ that proves the result. 
\end{proof}

\begin{thm}\label{description-s(M)-321}
If $\mathcal{M}=(M_1,\dots,M_q)$ is a sequence of commuting diagonalizable matrices, $({\bf m}^1=\Delta_{M_1}(S), \dots, {\bf m}^q=\Delta_{M_q}(S))$ is a star sequence of induced vectors, $S\in \mathfrak{S}(\mathcal{M})$, then the following statements hold.
\begin{enumerate}[(i)]
\item If $({\bf n}^1, \dots, {\bf n}^q)$ is a sequence of induced vectors, then there is $\sigma\in \mathcal{S}_n$ such that ${\bf n}^j={\bf m}^j_{\sigma}$, $\forall j\in \{1, \dots, q\}$.  
\item $\mathfrak{S}(\mathcal{M})=\{S(Y_{\boldsymbol{\alpha}^q_1}\oplus \dots \oplus Y_{\boldsymbol{\alpha}^q_{\mathfrak{c}(q)}})P_{\sigma}|\sigma\in \mathcal{S}_n,Y_{\boldsymbol{\alpha}^q_p} \in \mathcal{M}^{\emph{inv}}_{k_{{\boldsymbol{\alpha}^q_p}}}(\C)\}$. For all $j\in \{1, \dots, q\}$ we have $\Delta_{M_j}(S(Y_{\boldsymbol{\alpha}^q_1}\oplus \dots \oplus Y_{\boldsymbol{\alpha}^q_{\mathfrak{c}(q)}})P_{\sigma})={\bf m}^j_{\sigma}$.
\end{enumerate}
\end{thm}

\begin{proof}
$(i)$ There is $V\in \mathfrak{S}(\mathcal{M})$ such that, $\forall j\in \{1, \dots, q\}$, we have ${\bf n}^j=\Delta_{M_j}(V)$. 
$\mathcal{M}_r=(M_1,\dots, M_r)$, $r\in \{1, \dots, q\}$, is a sequence of commuting diagonalizable matrices, $({\bf m}^1, \dots, {\bf m}^r)$ is a star sequence of induced vectors by $\mathcal{M}_r$, and $({\bf n}^1, \dots,$ ${\bf n}^r)$ is a sequence of induced vectors by $\mathcal{M}_r$. 
We prove by induction the existence of $\sigma_r\in \mathcal{S}_n$ such that ${\bf n}^j={\bf m}^j_{{\sigma}_r}$, $\forall j\in \{1, \dots, r\}$.

For $r=1$ the announced result is proved in Lemma \ref{U-P-234}. We assume that there exists $\sigma_r\in \mathcal{S}_n$ such that ${\bf n}^j={\bf m}^j_{{\sigma}_r}$, $\forall j\in \{1, \dots, r\}$. From Lemma \ref{U-P-234} we deduce that ${\bf m}^1=\Delta_{M_1}(VP_{\sigma_r^{-1}}), \dots, {\bf m}^r=\Delta_{M_r}(VP_{\sigma_r^{-1}})$.  From Theorem \ref{commute-set-1001} we have that $VP_{\sigma_r^{-1}}= S(Y_{\boldsymbol{\alpha}^{r}_1}\oplus \dots \oplus Y_{\boldsymbol{\alpha}^{r}_{\mathfrak{c}(r)}})$ with $Y_{\boldsymbol{\alpha}^r_p} \in \mathcal{M}^{\text{inv}}_{k_{{\boldsymbol{\alpha}^r_p}}}(\C)$.

We remember that ${\bf m}^{r+1}=({\bf m}^{r+1,\boldsymbol{\alpha}^{r}_1}, \dots, {\bf m}^{r+1,\boldsymbol{\alpha}^{r}_{\mathfrak{c}(r)}})$, where ${\bf m}^{r+1, \boldsymbol{\beta}}$, $\boldsymbol{\beta}\in I_r$, is a star vector in $\C^{k_{\boldsymbol{\beta}}}$.
We note by ${\bf p}:={\bf n}^{r+1}_{\sigma_r^{-1}}=\Delta_{M_{r+1}}(VP_{\sigma_r^{-1}})$ and we write ${\bf p}=({\bf p}^{\boldsymbol{\alpha}^{r}_1}, \dots, {\bf p}^{\boldsymbol{\alpha}^{r}_{\mathfrak{c}(r)}})$, where ${\bf p}^{\boldsymbol{\beta}}\in \C^{k_{\boldsymbol{\beta}}}$, $\boldsymbol{\beta}\in I_r$. 
From Theorem \ref{commute-set-1001} we deduce that ${\bf p}^{\boldsymbol{\beta}}$ and ${\bf m}^{r+1,\boldsymbol{\beta}}$ have the same components (possibly in a different order). There exists $\tau_{\boldsymbol{\beta}}\in \mathcal{S}_{k_{\boldsymbol{\beta}}}$ such that ${\bf p}^{\boldsymbol{\beta}}={\bf m}^{r+1,\boldsymbol{\beta}}_{\tau_{\boldsymbol{\beta}}}$. If $\tau=\tau_{\boldsymbol{\alpha}^r_1}\oplus \dots \oplus \tau_{\boldsymbol{\alpha}^r_{\mathfrak{c}(r)}}\in \mathcal{S}_n$, then ${\bf n}^{r+1}={\bf p}_{\sigma_r}={\bf m}^{r+1}_{\tau\circ \sigma_r}$.

For $j\leq r$ we can write $\tau=\tau_{\boldsymbol{\alpha}^j_1}\oplus \dots \oplus \tau_{\boldsymbol{\alpha}^j_{\mathfrak{c}(j)}}$ with $\tau_{\boldsymbol{\alpha}^j_i}\in \mathcal{S}_{k_{\boldsymbol{\alpha}^j_i}}$. 
Using Theorem \ref{commute-set-1001} we have that ${\bf m}^j_{\tau}={\bf m}^j$. We deduce that ${\bf m}^j_{\tau\circ \sigma_r}={\bf m}_{\sigma_r}^j={\bf n}^j$.

$(ii)$ Let be $W\in \mathfrak{S}(\mathcal{M})$. From $(i)$ there is $\sigma\in \mathcal{S}_n$ such that ${\bf m}^j=\Delta_{M_j}(WP_{\sigma^{-1}})$, $\forall j\in \{1, \dots, q\}$. From Theorem \ref{commute-set-1001} we have that $WP_{\sigma^{-1}}=S(Y_{\boldsymbol{\alpha}^q_1}\oplus \dots \oplus Y_{\boldsymbol{\alpha}^q_{\mathfrak{c}(q)}})$, where $Y_{\boldsymbol{\alpha}^q_p} \in \mathcal{M}^{\text{inv}}_{k_{{\boldsymbol{\alpha}^q_p}}}(\C)$. We obtain that $W$ has the announced form. 
From hypotheses, Theorem \ref{commute-set-1001}, and Lemma \ref{U-P-234} we have that the matrix $S(Y_{\boldsymbol{\alpha}^q_1}\oplus \dots \oplus Y_{\boldsymbol{\alpha}^q_{\mathfrak{c}(q)}})P_{\sigma}$, with $Y_{\boldsymbol{\alpha}^q_p} \in \mathcal{M}^{\text{inv}}_{k_{{\boldsymbol{\alpha}^q_p}}}(\C)$, is an element of  $\mathfrak{S}(\mathcal{M})$.
\end{proof}

\begin{thm}\label{existence-principal-induced-444}
If $\mathcal{M}=(M_1,\dots,M_q)$ is a sequence of commuting diagonalizable matrices, then there exists $S\in \mathfrak{S}(\mathcal{M})$ such that $(\Delta_{M_1}(S), \dots, \Delta_{M_q}(S))$ is a star sequence of induced vectors.
\end{thm}

\begin{proof}
We prove the existence of $S_j\in \mathfrak{S}(\mathcal{M})$ such that $(\Delta_{M_1}(S_{j})$, ..., $\Delta_{M_j}(S_j))$ is a star sequence of induced vectors by $\{M_1, \dots, M_j\}$, $j\in \{1, \dots, q\}$. 

For $j=1$ we apply Lemma \ref{U-P-234}. We assume that the announced property is valid for $j-1$. If ${\bf m}^j=\Delta_{M_j}(S_{j-1})$, then we write ${\bf m}^j=({\bf m}^{j,\boldsymbol{\alpha}^{j-1}_1}, \dots, {\bf m}^{j,\boldsymbol{\alpha}^{j-1}_{\mathfrak{c}(j-1)}})$  with ${\bf m}^{j,\boldsymbol{\alpha}^{j-1}_i}\in \C^{k_{\boldsymbol{\alpha}^{j-1}_i}}$. For $i\in \{1, \dots, \mathfrak{c}({j-1})\}$ there is $\sigma_i\in \mathcal{S}_{k_{\boldsymbol{\alpha}^{j-1}_i}}$ such that ${\bf m}^{j,\boldsymbol{\alpha}^{j-1}_i}_{\sigma_i}$ is a star vector. We construct $\sigma=\sigma_1\oplus\dots\oplus \sigma_{\mathfrak{c}(j-1)}\in \mathcal{S}_n$ and $S_j=S_{j-1}P_{\sigma}$. By construction we have that $\Delta_{M_j}(S_{j})$ is a star vector and $\Delta_{M_r}(S_{j})=\Delta_{M_r}(S_{j-1})$, for $r<j$.   
\end{proof}

\begin{thm}\label{permutation-vectors-221}
Let $\mathcal{M}=(M_1,\dots,M_q)$ be a sequence of commuting diagonalizable matrices and $S,V\in \mathfrak{S}(\mathcal{M})$. If ${\bf m}^j=\Delta_{M_j}(S)$, ${\bf n}^j=\Delta_{M_j}(V)$, $\forall j\in \{1,\dots, q\}$, then there exists $\sigma\in \mathcal{S}_n$ such that ${\bf n}^j={\bf m}^j_{\sigma}$,  $\forall j\in \{1,\dots, q\}$.
\end{thm}

\begin{proof}
From Theorem \ref{existence-principal-induced-444} there exists $({\bf p}^1, \dots, {\bf p}^q)$ a star sequence of induced vectors. From Theorem \ref{description-s(M)-321} there are $\tau, \mu \in \mathcal{S}_n$ such that, $\forall j\in \{1, \dots, q\}$ we have ${\bf m}^j={\bf p}^j_{\tau}$ and ${\bf n}^j={\bf p}^j_{\mu}$. Consequently, ${\bf n}^j={\bf m}^j_{\tau^{-1}\circ \mu}$.
\end{proof}

\subsection{A method for finding  a star sequence of two induced vectors without a diagonalizing matrix}\label{Gamma-without-S}

We consider $\mathcal{M}:=(A,B)$ a sequence of diagonalizable matrices in $\mathcal{M}_n(\C)$ such that $AB=BA$.  We present a method to find a star sequence of induced vectors by $\mathcal{M}$. This method assumes the possibility to calculate the eigenvalues of a combinations of the matrices $A$ and $B$. 

First, we set the star vector ${\bf a}=\text{vec}(\lambda_1 {\bf 1}_{k_1}, \dots, \lambda_d {\bf 1}_{k_d})$, where $\lambda_1, \dots, \lambda_d$ are the different eigenvalues of $A$ and $k_1, \dots, k_d$ is their algebraic multiplicity. This is an induced vector by $A$. If $({\bf a}, {\bf b})$ and $({\bf a}, \widehat{\bf b})$ are sequences of induced vectors by $(A,B)$, where ${\bf b}=\text{vec}(\boldsymbol{\mathfrak{b}}_1,\dots,\boldsymbol{\mathfrak{b}}_d), \widehat{\bf b}=\text{vec}(\widehat{\boldsymbol{\mathfrak{b}}}_1,\dots,\widehat{\boldsymbol{\mathfrak{b}}}_d)\in \C^{k_1}\times \dots \times \C^{k_d}$, then $\boldsymbol{\mathfrak{b}}_j$ and $\widehat{\boldsymbol{\mathfrak{b}}}_j$ have the same components (possibly in a different order, see Theorem \ref{commute-set-1001}). It is sufficient to find the components of each vector $\boldsymbol{\mathfrak{b}}_j$.

 We compute $b_1, \dots, b_n$ the eigenvalues of $B$ and construct the set 
\begin{equation}\label{set-B-44}
\mathcal{B}:=\left\{\frac{\lambda_s b_i-\lambda_r b_j}{\lambda_r-\lambda_s}\in \C\,|\,i,j\in \{1,\dots, n\}, i\neq j, \,\text{and}\, r,s\in \{1, \dots, d\}, r\neq s\right\}.
\end{equation}

 If $\lambda_q\neq 0$ and $\beta\in \C\backslash \mathcal{B}$, then the eigenvalues of the diagonalizable matrix $\frac{1}{\lambda_q}(AB+\beta A)-\beta I_n$ are the components of the vector
$$\Big(\frac{\lambda_1}{\lambda_q}\boldsymbol{\mathfrak{b}}_1+\frac{\beta (\lambda_1-\lambda_q)}{\lambda_q}{\bf 1}_{k_1}, \dots, \boldsymbol{\mathfrak{b}}_q, \dots, \frac{\lambda_d}{\lambda_q}\boldsymbol{\mathfrak{b}}_d+\frac{\beta (\lambda_d-\lambda_q)}{\lambda_q}{\bf 1}_{k_d}\Big).$$
The components of $\boldsymbol{\mathfrak{b}}_q$ are terms of the sequence $b_1, \dots, b_n$. Because $\beta\in \C\backslash \mathcal{B}$ we deduce that a component of the vector $\frac{\lambda_r}{\lambda_q}\boldsymbol{\mathfrak{b}}_r+\frac{\beta (\lambda_r-\lambda_q)}{\lambda_q}{\bf 1}_{k_r}$, with $q\neq r$, is not a term of the sequence $b_1, \dots, b_n$. We obtain that {\it the components of $\boldsymbol{\mathfrak{b}}_q$ are the common terms of the sequence $(b_1,\dots,b_n)$ and a sequence formed with the eigenvalues of the matrix $\frac{1}{\lambda_q}(AB+\beta A)-\beta I_n$}. 
 
 The above result allows the computation of the components of the vectors $\boldsymbol{\mathfrak{b}}_q$ that correspond to non-zero eigenvalues.  If $\lambda_{q^*}=0$, then the components of the vector $\boldsymbol{\mathfrak{b}}_{q^*}$ are the components that remain by eliminating the components of the vectors $\boldsymbol{\mathfrak{b}}_q$, $q\neq q^*$, among the eigenvalues of $B$.

\section{Examples}

\begin{exmp}
We consider the homogeneous linear matrix equation
$$AXB+2X=O_3,\,\,\,A=\begin{pmatrix} 0 & 1 & 0 \\ 1 & 0 & 0 \\ 0 & 0 & 1 \end{pmatrix},\,\,B=\begin{pmatrix} 1 & -1 & 0 \\ -1 & 1 & 0 \\ 0 & 0 & 2 \end{pmatrix}.$$
This equation has the form \eqref{linear-matrix-equation} with $n=3$, $k=2$, $A_1=A$, $B_1=B$, $A_2=2I_3$, $B_2=I_3$, and $C=O_3$. Since all these matrices are real symmetric matrices we get that they are normal matrices. It is easy to verify that $\mathcal{M}=\{A_1,A_2,B_1,B_2,C\}$ is a commuting set of matrices.
The above equation is consistent and the set of solutions is a vectorial subspace of $\mathcal{M}_n(\C)$. In what follows we use the method presented in Section \ref{Gamma-without-S} to compute a sequence of induced vectors by $(A,B)$. The eigenvalues of $A$ are $\lambda_1=1$ with the algebraic multiplicity $k_1=2$ and $\lambda_2=-1$ with the algebraic multiplicity $k_2=1$. We construct the star vector ${\bf a}=\text{vec}(1,1,-1)$. The eigenvalues of $B$ are $2, 2, 0$ and the set defined in \eqref{set-B-44} is $\mathcal{B}=\{-2,-1\}$. To determine the order of the components of ${\bf b}$ we choose $\beta=0\notin \mathcal{B}$. The first two components of ${\bf b}$ are the common elements of the sequence of eigenvalues of $\frac{1}{\lambda_1}AB$ and of the sequence of eigenvalues of $B$. The matrix $\frac{1}{\lambda_1}AB$ has $0, 2, -2$ as the eigenvalues. If we choose ${\bf b}=\text{vec}(0, 2, 2)$, then $({\bf a}, {\bf b})$ is a sequence of induced vectors by $(A,B)$. It is easy to observe that $({\bf a}^1, {\bf b}^1, {\bf a}^2, {\bf b}^2, {\bf c})$ is a sequence of induced vectors by $(A_1,B_1,A_2,B_2,C)$, where ${\bf a}^1={\bf a}$, ${\bf b}^1={\bf b}$, ${\bf a}^2=\text{vec}(2,2,2)$, ${\bf b}^2=\text{vec}(1,1,1)$, and ${\bf c}=\text{vec}(0,0,0)$.
The relevant matrix defined by this vectors is 
$$\Gamma={\bf a}^1 ({\bf b}^1)^T+{\bf a}^2 ({\bf b}^2)^T=\begin{pmatrix} 2 & 4 & 4 \\ 2 & 4 & 4 \\ 2 & 0 & 0 \end{pmatrix}.$$
We deduce that the dimension of the space of the solutions is equal with 2. 

By computation we obtain that $S=\begin{pmatrix} 1 & 0 & -1 \\ 1 & 0 & 1 \\ 0 & 1 & 0\end{pmatrix}\in \mathfrak{S}(\mathcal{M})$ and a solution of the
 equation has the form 
$X= \begin{pmatrix} \alpha & -\alpha & -\beta \\ -\alpha & \alpha & \beta \\ 0 & 0 & 0\end{pmatrix},\,\,\alpha,\beta\in \C.$

\end{exmp}\medskip

\begin{exmp}\label{continuous-Lyapinov-not-normal-23}
We consider the homogeneous continuous Lyapunov equation 
$$A^{\star}X+XA=O_2,\,\,\,A=\begin{pmatrix} 0 & 1 \\ 0 & 0 \end{pmatrix}.$$
$A$ is not a normal matrix and $0$ is an eigenvalue with the algebraic multiplicity equal to 2. The solutions are $X=\begin{pmatrix} 0 & a \\ -a & b \end{pmatrix}$ with $a,b\in \C$.
The vectorial subspace of the solutions has the dimension equal to 2. The cardinality of the set $\{(r,s)\in \{1,\dots, n\}^2\,|\,\overline{a}_r+a_s= 0\}$ is equal to 4.  This example shows that the result given in Theorem \ref{continuous-Lyapunov}-$(iii)$ can be false when $A$ is not a normal matrix. 
\end{exmp} \medskip

\begin{exmp}\label{discrete-Lyapunov-not-normal-45}
We consider the homogeneous discrete Lyapunov equation 
$$A^{\star}XA-X=O_2,\,\,\,A=\begin{pmatrix} 1 & 1 \\ 0 & 1 \end{pmatrix}.$$
$A$ is not a normal matrix and $1$ is an eigenvalue with the algebraic multiplicity equal to 2. The solutions are $X=\begin{pmatrix} 0 & a \\ -a & b \end{pmatrix}$ with $a,b\in \C$.
The vectorial subspace of the solutions has the dimension equal to 2. The cardinality of the set $\{(r,s)\in \{1,\dots, n\}^2\,|\,\overline{a}_r a_s= 1\}$ is equal to 4.  This example shows that the result given in Theorem \ref{discrete-Lyapunov}-$(iii)$ can be false when $A$ is not a normal matrix. 
\end{exmp} \medskip

\begin{exmp}\label{exmp-stein-1}
We consider the Stein equation
$$AXA-X=I_2, \,\,\,A=\begin{pmatrix} 1 & 1 \\ 0 & 1 \end{pmatrix}.$$
$A$ and $I_2$ commute, but $A$ is not a diagonalizable matrix. Because $A^2-I_2$ is a nilpotent matrix we obtain that  
$\widehat{X}=(A^2-I_2)^{\bf D}I_2=O_2$
and it is not a solution of the above Stein equation. The above Stein equation is consistent having the unique solution $X=\begin{pmatrix} 2 & -3 \\ 0 & 2 \end{pmatrix}$. The attached standard linear matrix equation  is $(A^2-I_2)X=I_2$ and it is inconsistent. 
\end{exmp}\medskip

\begin{exmp}\label{exmp-stein-2}
We consider the Stein equation
$$AXA-2X=C, \,\,\,A=\begin{pmatrix} 1 & 1 \\ 1 & -1 \end{pmatrix},\,\,\,C=\begin{pmatrix} 0 & 1 \\ -1 & 0 \end{pmatrix}.$$
All parameters are normal matrices, but $AC\neq CA$. We have the equality 
$\widehat{X}=(A^2-2I_2)^{\boldsymbol{\dagger}}C=O_2$
and it is not a solution of the above Stein equation. This is a consistent equation and its solutions are $X=\begin{pmatrix} 2a+b-\frac{1}{2} & a-\frac{1}{2} \\ a & b \end{pmatrix}$ with $a,b\in \C$. The attached standard linear matrix equation  is inconsistent. 
\end{exmp}

\appendix 

\section{Permutation matrices}\label{permutation-matrix-22}

A permutation of the set $\{1,2, \dots,n\}$ is a bijection from this set to itself. The set of all permutations is noted by $\mathcal{S}_n$.

If $n=k_1+\dots +k_d$ with $k_1, \dots, k_d\in \mathbb{N}^*$ and $\sigma_1\in \mathcal{S}_{k_1}$, ..., $\sigma_d\in \mathcal{S}_{k_d}$, then $\sigma:=\sigma_{k_1}\oplus\dots\oplus\sigma_{k_d}\in \mathcal{S}_n$ is given by
\begin{equation}\label{sigma-oplus-432}
\sigma(i)=\begin{cases} 
\sigma_1(i) & \text{if} \,\,1\leq i\leq k_1 \\ 
\sigma_2(i)+k_1 & \text{if} \,\,k_1+1\leq i\leq k_1+k_2 \\
\dots & \dots \\
\sigma_d(i)+k_1+\dots+k_{d-1} & \text{if} \,\,k_1+\dots+k_{d-1}+1\leq i\leq n.
\end{cases}
\end{equation}

$P\in \mathcal{M}_n(\C)$ is a permutation matrix if it has exactly one entry of $1$ in each row and each column and all other entries $0$. We note by $\mathcal{P}_n(\C)$ the set of $n\times n$ permutation matrices. The function\footnote{${\bf e}_1$, ..., ${\bf e}_n$ are the vectors of the canonical basis in $\C^n$.} $\sigma\in \mathcal{S}_n\to P_{\sigma}:=({\bf e}_{\sigma(1)},\dots, {\bf e}_{\sigma(n)})\in \mathcal{P}_n(\C)$ is a bijection. The components of $P_{\sigma}$ are $(P_{\sigma})_{ij}=\delta_{i\sigma(j)}$.

If $\sigma, \tau\in \mathcal{S}_n$, then we have the following equality $P_{\sigma \circ \tau}=P_{\sigma}P_{\tau}$. The inverse of $P_{\sigma}$ is $P^{-1}_{\sigma}=P^{\star}_{\sigma}=P_{\sigma^{-1}}$. A permutation matrix is an unitary matrix.

If $\sigma\in \mathcal{S}_n$ and $M=(m_{ij})\in \mathcal{M}_n(\C)$, then we have $(P_{\sigma}M)_{ij}=m_{\sigma^{-1}(i)j}$, $(MP_{\sigma})_{ij}=m_{i\sigma(j)}$, and $(P_{\sigma}^{-1}MP_{\sigma})_{ij}=(P_{\sigma}^{\star}MP_{\sigma})_{ij}=m_{\sigma(i)\sigma(j)}$.

If $\sigma_1\in \mathcal{S}_{k_1}$, ..., $\sigma_d\in \mathcal{S}_{k_d}$, and $n=k_1+\dots +k_d$, then $P_{\sigma}=P_{\sigma_1}\oplus P_{\sigma_2} \oplus \dots \oplus P_{\sigma_d}$, where $\sigma\in \mathcal{S}_n$ is defined by \eqref{sigma-oplus-432}. 

\section{The Moore-Penrose inverse. Drazin inverse. The group inverse}\label{Moore-Penrose-prop}

If $A\in \mathcal{M}_{m\times n}(\C)$, then {\bf the Moore-Penrose inverse} of $A$, see \cite{stoer}, is the unique matrix $A^{\boldsymbol{\dagger}}\in \mathcal{M}_{n\times m}(\C)$ which satisfies the following properties: $A^{\boldsymbol{\dagger}}A=(A^{\boldsymbol{\dagger}}A)^{\star}$, $AA^{\boldsymbol{\dagger}}=(AA^{\boldsymbol{\dagger}})^{\star}$,  $AA^{\boldsymbol{\dagger}}A=A$, and $A^{\boldsymbol{\dagger}}AA^{\boldsymbol{\dagger}}=A^{\boldsymbol{\dagger}}$.
The Moore-Penrose inverse has the following properties: $O_{m\times n}^{\boldsymbol{\dagger}}=O_{n\times m}$, $(\text{diag}(\lambda_1,\dots,\lambda_n))^{\boldsymbol{\dagger}}=\text{diag}(\lambda_1^{\boldsymbol{\dagger}},\dots,\lambda_n^{\boldsymbol{\dagger}})$, $(A^{\star})^{\boldsymbol{\dagger}}=(A^{\boldsymbol{\dagger}})^{\star}$, $A^{\boldsymbol{\dagger}}=A^{-1}$ when $A\in \mathcal{M}^{\text{inv}}_{n}(\C)$. 

If $A\in\mathcal{M}_n(\C)$, then $\text{ind}(A)$ - the index of $A$ -
 is the least nonnegative integer such that $\text{rank}(A^{\text{ind}(A)+1}) =\text{rank}(A^{\text{ind}(A)})$. If $A$ is a diagonalizable matrix, then $\text{ind}(A)\in \{0,1\}$.

{\bf The Drazin inverse} of $A\in\mathcal{M}_n(\C)$ is the unique matrix $A^{\bf D}\in\mathcal{M}_n(\C)$ that satisfies: $A^{\text{ind}(A)+1}A^{\bf D}=A^{\text{ind}(A)}$, $A^{\bf D}AA^{\bf D}=A^{\bf D}$, and $A^{\bf D}A=AA^{\bf D}$.

If $A\in\mathcal{M}_n(\C)$ is invertible, then $A^{\bf D}=A^{-1}$. If ${\bf a}=(a_1, \dots, a_n)\in \C^n$, then $(\text{diag}({\bf a}))^{\bf D}=(\text{diag}({\bf a}))^{\boldsymbol{\dagger}}=\text{diag}(a_1^{\boldsymbol{\dagger}}, \dots, a_n^{\boldsymbol{\dagger}})$.
Every matrix $A\in\mathcal{M}_n(\C)$ can be written in the form $A=S (B\oplus N) S^{-1}$, where $S\in \mathcal{M}^{\text{inv}}_n(\C)$, $B\in \mathcal{M}^{\text{inv}}_m(\C)$ ($m\in \{0,1,\dots,n\}$), and $N\in \mathcal{M}_{n-m}(\C)$ is a nilpotent matrix. If $A$ has the above representation, then $A^{\bf D}=S (B^{-1}\oplus O_{n-m}) S^{-1}$. The Drazin inverse $A^{\bf D}$ is a polynomial in $A$. If $A,B\in\mathcal{M}_n(\C)$ such that $AB=BA$, then $A^{\bf D}B=BA^{\bf D}$ and $A^{\bf D}B^{\bf D}=B^{\bf D}A^{\bf D}$. 

If $A\in\mathcal{M}_n(\C)$ and $W\in\mathcal{M}_n(\C)$ is invertible, then $(WAW^{-1})^{\bf D}=WA^{\bf D}W^{-1}$.

If $A\in\mathcal{M}_n(\C)$ and $\text{ind}(A)\in \{0,1\}$, then $A^{\bf D}$, the Drazin inverse of $A$, is also called {\bf the group inverse} of $A$ and is denoted by $A^{\#}$ (see \cite{mayer}).

\section{Normal matrices}

The matrix $M\in \mathcal{M}_n(\C)$ is a normal matrix if $MM^{\star}=M^{\star}M$, where $M^{\star}\in \mathcal{M}_{n}(\C)$ is the conjugate transpose of $M$. An impressive number of properties of normal matrices are presented in the paper \cite{grone}. A normal matrix is a diagonalizable matrix.

The matrix $M\in \mathcal{M}_n(\C)$ is a normal matrix if and only if there is $U\in \mathfrak{S}(M)\cap\mathcal{U}(n)$, where
$\mathcal{U}(n)=\{U\in \mathcal{M}_n(\C)\,|\,UU^{\star}=U^{\star}U=I_n\}$
is the set of $n$-by-$n$ unitary matrices.
If $\mathcal{S}$ is a set of normal matrices, then
$\mathfrak{S}(\mathcal{S})\cap \mathcal{U}(n)\neq \emptyset$ if and only if $\mathcal{S}$ is a commuting set. 

\begin{lem}\label{normal-Drazin-Moore-Penrose}
If $M$ is a normal matrix, $U\in \mathfrak{S}(M)\cap\mathcal{U}(n)$, and $\Delta_M(U):=\emph{vec}(m_1,\dots, m_n)\in \C^n$, then 
$M^D=M^{\boldsymbol{\dagger}}=U\emph{diag}(m^{\boldsymbol{\dagger}}_1, \dots, m^{\boldsymbol{\dagger}}_n)U^{\star}.$
\end{lem}

\begin{proof}
Because $M$ is a diagonalizable matrix we have that $\text{ind}(M)\in \{0,1\}$.
If $X=U\text{diag}(m^{\boldsymbol{\dagger}}_1, \dots, m^{\boldsymbol{\dagger}}_n)U^{\star}$, then $MX  =U\text{diag}(m_1m^{\boldsymbol{\dagger}}_1, \dots, m_nm^{\boldsymbol{\dagger}}_n)U^{\star}=XM$, $(MX)^{\star} =U(\text{diag}(m_1m^{\boldsymbol{\dagger}}_1, \dots, m_nm^{\boldsymbol{\dagger}}_n))^{\star}U^{\star}=MX$, $MXM =U\text{diag}(m^2_1m^{\boldsymbol{\dagger}}_1,$ $\dots, m^2_nm^{\boldsymbol{\dagger}}_n)U^{\star}=M$, $XMX =U\text{diag}(m_1(m^{\boldsymbol{\dagger}}_1)^2, \dots, m_n(m^{\boldsymbol{\dagger}}_n)^2)U^{\star}=X$.
We observe that $X$ verifies the definition of the Moore-Penrose inverse and the definition of the Drazin inverse.  From uniqueness of the Moore-Penrose inverse and the Drazin inverse one obtains the announced result. 
\end{proof}

It is easy to observe the following results. 
\begin{lem}\label{prop-normal-Moore}
Let $\mathcal{S}\subset \mathcal{M}_n(\C)$ be a commuting family of normal matrices. If $\alpha\in \C$ and $M,N\in \mathcal{S}$, then $\alpha M$, $M^{\boldsymbol{\dagger}}$, $M+N$, and $MN$ are normal matrices that commute with all matrices in $\mathcal{S}$. Also, we have $(MN)^{\boldsymbol{\dagger}}=M^{\boldsymbol{\dagger}}N^{\boldsymbol{\dagger}}=N^{\boldsymbol{\dagger}}M^{\boldsymbol{\dagger}}$. 
\end{lem}

\end{document}